\documentclass[11pt]{amsart}
\usepackage{eurosym}
\usepackage{amsfonts}

\usepackage{graphicx}%
\usepackage{multirow}%
\usepackage{amsmath,amssymb,amsfonts}%
\usepackage{amsthm}%
\usepackage{mathrsfs}%
\usepackage[title]{appendix}%
\usepackage{xcolor}%
\usepackage{textcomp}%
\usepackage{manyfoot}%
\usepackage{booktabs}%
\usepackage{algorithm}%
\usepackage{algorithmicx}%
\usepackage{algpseudocode}%
\usepackage{listings}%
\usepackage{cleveref}
\DeclareMathAlphabet{\pazocal}{OMS}{zplm}{m}{n}

\newcommand{\Ub}{\pazocal{U}}
\newcommand{\Zb}{\pazocal{Z}}
\newcommand{\Hb}{\pazocal{H}}
\newcommand{\Pb}{\pazocal{P}}
\newcommand{\Fb}{\pazocal{F}}
\newcommand{\Db}{\pazocal{D}}
\newtheorem{theorem}{Theorem}[section]

\newtheorem{remark}{Remark}

\numberwithin{equation}{section}

\begin{document}
\title[Stratified equatorial flows in cylindrical coordinates with surface tension]{Stratified equatorial flows in cylindrical coordinates with surface tension}
\author[G. Cristina]{G. Cristina}
\address{G. Cristina, Faculty of Mathematics and Computer Science, Babeș-Bolyai
University,
400084 Cluj-Napoca, Romania}
\email{cristina.gheorghe@ubbcluj.ro}
\author[A. Stan]{Andrei Stan}
\address{A. Stan, Faculty of Mathematics and Computer Science, Babeș-Bolyai
University,
400084 Cluj-Napoca, Romania \& Tiberiu Popoviciu Institute of Numerical
Analysis, Romanian Academy, P.O. Box 68-1, 400110 Cluj-Napoca, Romania}
\email{andrei.stan@ubbcluj.ro}

\begin{abstract}
  This paper considers a mathematical model of steady flows of an inviscid and incompressible fluid moving in the azimuthal direction. The water density varies both in depth and latitude and the waves are propagating under the force of gravity, over a flat bed and with a free surface, on which acts a force of surface tension. Our solution pertains to large scale equatorial dynamics of a fluid with free surface expressed in cylindrical coordinates. We also prove a regularity result for the free surface.
\end{abstract}

\maketitle

\bigskip



\section{Introduction}

Our aim is to study equatorial flows
exhibiting general (continuous) stratification depending on depth and
latitude. This work is motivated by the intriguing features observed in the Equatorial region, situated approximately 2° latitude from the Equator. In this area, strong vertical ocean stratification prevails, surpassing other ocean zones and creating a distinct thermocline—a sharp interface separating warmer, less dense surface water from colder, denser water below (see \cite{const2016steady, feodorov2009eq,izumo2005,johnson2001velocity, Kessler1995OceanicEW,mccreary1985modeling}). Additionally, the equatorial region showcases underlying currents, including a westward drift near the surface triggered by prevailing trade winds and the presence of the Equatorial Undercurrent (EUC), an eastward pointing stream residing on the thermocline. These characteristics contribute to the complexity of large-scale ocean motions, emphasizing the importance of understanding fluid density variations in both depth and latitude, primarily influenced by temperature and, to a lesser extent, salinity (see, e.g., \cite{Kessler1995OceanicEW, garrison2009essentials,mccreary1985modeling}). 

Our goal in this study is to establish the existence of an exact solution within a cylindrical frame to the highly nonlinear and intractable equations of geophysical fluid dynamics (GFD). In this context, we would like to mention that, recently, a number of exact solutions of the GFD governing equations were devised, cf. \cite{const2012exact, const2013nonlinear, const2014nonhydrostatic,const2015dynamicswaves, const2016steady, const2016Antarctic, const2017Pacific, Matioc_2012, HENRY20196788,henry2019azimuthal,hsu2015exact, IONESCUKRUSE2015190}. The solution characterizes a flow propagating in the azimuthal direction, influenced by both gravity and surface tension, while maintaining variability in density across all directions. This form of density distribution proves sufficiently general to capture and model most geophysical fluid dynamics phenomena \cite{vallis2017atmospheric}.

The Equatorial motion has been the subject of previous studies, where has been proven the existence of exact solutions  under various contexts (see, e.g., \cite{Henry,Hsu, const2012modeling,const2012exact,const2013nonlinear}). The novelty of this paper lies in deducing an existence result, considering both surface tension and density variation. While comparable results, utilizing cylindrical coordinates, have been achieved in   \cite{Henry} 
 and \cite{Hsu}, it is noteworthy that these outcomes are exclusively applicable to one of the two cases mentioned earlier.

\section{Preliminaries}

 Considering the line of Equator "straightened" and replaced  to the z axis, and the body of the sphere, represented by a circular disc described by radius r and the angle of deflection from the equator, $\theta \in \displaystyle \left(-\frac{\pi}{2}, \frac{\pi}{2}\right)$, the fluid motion can be described in cylindrical coordinates. We can provide a right-handed rotating system of coordinates (r, $\theta$, z), where r measures the distance from the center of the disc (physically, it means the distance to the center of the Earth), $\theta$ is increasing from north to south and the z axis points positively from west to east. Note that $\theta=0$ corresponds to the line of Equator. The associated unit vectors in the (r, $\theta$, z) system are ($\vec{e_r}, \vec{e_{\theta}}, \vec{e_z}$) and the velocity corresponding components are (u,v,w). \\
\indent The equations describing the motion of an inviscid and incompressible flow  are comprised in the  Euler equations and, respectively,  mass conservation equation. In  cylindrical coordinates, they are given by
\begin{equation}\label{eq:1}
\begin{cases}
\displaystyle u_t+uu_r+\frac{v}{r}u_{\theta}+w u_z-\frac{v^2}{r}=-\frac{1}{\rho}p_r+F_r\\
~\\
\displaystyle v_t+uv_r+\frac{v}{r} v_\theta+wvz+\frac{uv}{r}=-\frac{1}{\rho}\frac{1}{r}p_\theta+F_\theta\\
~\\
\displaystyle w_t+uw_r+\frac{v}{r}+w_\theta+ww_z=-\frac{1}{\rho}p_z+F_z,
\end{cases}
\end{equation}
and \begin{equation} \label{eq:2}
\frac{1}{r}\frac{\partial }{\partial r} (r \rho u)+\frac{1}{r}\frac{\partial}{\partial \theta}(\rho v)+\frac{\partial}{\partial z}(\rho w)=0,
\end{equation}
respectively. We denote $\rho$ as the fluid density, $p$ as the pressure in the fluid, and ($F_r, F_\theta, F_z$) as the body-force vector, with these quantities depending on the variables $r$, $\theta$, and $z$.

To ensure a precise examination of fluid dynamics in specific locations, it is essential to account for the Earth's rotation effects. The additional terms to incorporate into the expression of the equations of Euler are related to the Coriolis force $ 2\vec{\Omega}\times \vec{u}$
and the centripetal acceleration $\vec{\Omega}\times (\vec{\Omega}\times\vec{r})$,
where
\begin{align*}
    &\vec{\Omega}=-\Omega((\sin \theta) \vec{e_r}+(\cos \theta)\vec{e_\theta}),\\&
\vec{u}=u \vec{e_r}+v\vec{e_\theta}+w\vec{e_z},\\&
\vec{r}=r\vec{e_r}.
\end{align*}
Here, $\Omega\approx 7.29 \times 10^{-5} rad/s$ stands for   the constant rate of rotation of the Earth.

The two forces mentioned earlier sum up in the Euler equations. Thus, since the gravity is given by the vector $(-g,0,0)$, $2\vec{\Omega}\times \vec{u}=2\Omega(-w\cos\theta, w\sin\theta, u\cos\theta-v\sin\theta)$ and $\vec{\Omega}\times(\vec{\Omega}\times\vec{r})=r\Omega^2\cos\theta(-\cos\theta,\sin\theta,0),$ one has

\begin{equation}\label{eq:5}
\begin{cases}
\displaystyle u_t+uu_r+\frac{v}{r}u_{\theta}+w u_z-\frac{v^2}{r}-2\Omega w\cos\theta-r\Omega^2\cos^2\theta=-\frac{1}{\rho}p_r-g\\
~\\
\displaystyle v_t+uv_r+\frac{v}{r} v_\theta+wvz+\frac{uv}{r}+2\Omega w\sin\theta+r\Omega^2\sin\theta\cos\theta=-\frac{1}{\rho}\frac{1}{r}p_\theta\\
~\\
\displaystyle w_t+uw_r+\frac{v}{r}+w_\theta+ww_z+2\Omega(u\cos\theta-v\sin\theta)=-\frac{1}{\rho}p_z.
\end{cases}
\end{equation}
Besides the mass conservation and Euler equations, the water motion is also subject to boundary conditions. Considering that the fluid extends to infinity in all horizontal directions, there are two boundaries: the rigid flat bed and the free surface of the water.  
On the bed $r=d(\theta,z)$, we have the  kinematic boundary condition \begin{equation}\label{eq:7}
u=w d_z+\displaystyle\frac{1}{r} v d_\theta.
\end{equation}
On the free surface $r=R+h(\theta,z)$, along with the kinematic boundary condition \begin{equation}\label{eq:6}
u=w h_z+\displaystyle\frac{1}{r} v h_\theta,
\end{equation}
we also have the dynamic condition \begin{equation}\label{eq:8}
p=P(\theta,z)+\sigma\nabla\cdot \vec{n},
\end{equation}
where $\sigma$ is the coefficient of surface tension and $\vec{n}$ is the unit normal vector, pointing outward. The function $h$ is unknown and will be determined later.

\section{Main result}



Our aim is to determine the pressure under the additional assumption that the flow is purely azimuthal and invariant in this direction, i.e.,
\begin{equation}\label{eq:9}
u=v=0 \text{ and } w=w(r,\theta).
\end{equation} 
Consequently, since the dependence in the "z" direction is eliminated, relation \eqref{eq:8} has the form
\begin{equation} \label{eq:10}
p=P(\theta)+\sigma\nabla\cdot \vec{n}.
\end{equation}
Additionally, both the free surface and the bed will be characterized by $r=R+h(\theta)$ and $r=d(\theta)$, respectively.
 \begin{remark}
 To account for geophysical factors in equatorial flows, the polar angle $\theta$ is confined to an interval of the form $\left[-\varepsilon,\varepsilon\right ]$.
Selecting  $\varepsilon=0.016$, it corresponds to a strip approximately 100 km wide centered around the equator.
 \end{remark}From  \eqref{eq:9}, the Euler equations becomes
 \begin{equation}\label{eq:11}
\begin{cases}
-2\Omega w\cos\theta-r\Omega^2\cos^2\theta=-\displaystyle \frac{1}{\rho}p_r-g\\
2\Omega w\sin\theta+r\Omega^2\sin\theta\cos\theta=\displaystyle  -\frac{1}{\rho}\frac{1}{r}p_\theta\\
\quad\indent\quad\indent\indent 0=-\displaystyle\frac{1}{\rho}p_z,
\end{cases}
\end{equation}
or equivalently 
\begin{equation}\label{eq:12}
\begin{cases}
\rho(r,\theta)\Omega\cos\theta(2w(r, \theta)+r\Omega\cos\theta)=p_r+g\rho(r, \theta)\\
\rho(r,\theta) r \Omega\sin\theta(2w(r, \theta)+r\Omega\cos\theta)=-p_\theta\\
\quad\indent\quad\indent\indent\quad\indent 0=p_z.
\end{cases}
\end{equation}
Letting $\Ub (r,\theta):=2w(r,\theta)+r\Omega\cos\theta$, if we derive the first equation with respect to $\theta$ and the second one with respect to r, we obtain
\[
\left\{ \begin{aligned} 
 \displaystyle (\rho \Omega\Ub\cos\theta)_\theta &= \displaystyle p_{r\theta}+g\rho_\theta\\
  \displaystyle(\rho r \Omega\Ub\sin\theta)_r &= \displaystyle -p_{\theta r}
\end{aligned} \right.
\]
Next, we sum up the two equations to deduce
\[(\rho\Omega\Ub\cos\theta)_\theta+(\rho r \Omega\Ub\sin\theta)_r=g\rho_\theta, \]
which yields
\begin{equation} \label{eq:14}
(r\sin\theta)\Zb_r(r,\theta)+(\cos\theta)\Zb_\theta(r,\theta)=(r\cos\theta)g\rho_\theta,
\end{equation}
where $\Zb(r, \theta):=\rho r\Omega\Ub\cos\theta$.

Following the method of characteristics, we look for solutions $s\rightarrow \tilde{r}(s)$ and $s\rightarrow\tilde{\theta}(s)$ such that
\begin{align}\label{eq:15}
& \widetilde{r}~'(s)=\widetilde{r}(s)\sin\widetilde{\theta}(s),\\&  \widetilde{\theta}'(s)=\cos\widetilde{\theta}(s). \nonumber
\end{align}
Thus,  relation  \eqref{eq:14} is equivalent to
\begin{equation}\label{eq:16}
\displaystyle\frac{d}{ds}\big(\Zb(\widetilde{r}(s),\widetilde{\theta}(s))=(\widetilde{r}(s)\cos\widetilde{\theta}(s))g\rho_\theta(\widetilde{r}(s_, \widetilde{\theta}(s)).
\end{equation}
Simple computations yield $\widetilde{r}''(s)=\widetilde{r}(s)$, leading to the solution \begin{equation}\label{eq:17}
\widetilde{r}(s)=c_1e^{s}+c_2e^{-s},
\end{equation}
where $c_1, c_2 $ are some real constants. Now, since  $\widetilde{r}^\prime(s)=c_1e^{s}-c_2e^{-s}$, we infer \begin{equation}\label{eq:18}
\widetilde{\theta}(s)=\arcsin\displaystyle\frac{c_1e^{s}-c_2e^{-s}}{c_1e^{s}+c_2e^{-s}}=\arcsin\displaystyle\frac{c e^{2s}-1}{c e^{2s}+1},
\end{equation}
where $c=\displaystyle\frac{c_1}{c_2}$ is constant.
Returning to our equation \eqref{eq:14}, we note that 
\begin{equation}\label{rcos constant}
\displaystyle\frac{d}{ds} (\widetilde{r}(s) \cos\widetilde{\theta}(s))=0.
\end{equation} 
Assuming, without loss of generality, that $\widetilde{\theta}(0)=0$, we find that $c=1$, implying $c_1=c_2$. For given values of $r$ and $\theta$, we seek $s_0 \in \mathbb{R}$ such that
\begin{equation}\label{eq:19}
\widetilde{\theta}(s_0)=\theta, \widetilde{r}(s_0)=r.
\end{equation}
The initial condition in \eqref{eq:19} leads to $\displaystyle\frac{e^{2s_0}-1}{e^{2s_0}+1}=\sin\theta$, thus
  \begin{equation}
  s_0 (\theta)= \displaystyle \frac{1}{2}\ln \displaystyle\frac{1+\sin\theta}{1-\sin\theta},
  \end{equation}
  while the second condition from \eqref{eq:19} leads to 
  \begin{equation*}
      c_1\bigg(\sqrt{\displaystyle\frac{1+\sin\theta}{1-\sin\theta}}+\sqrt{\displaystyle\frac{1-\sin\theta}{1+\sin\theta}}\bigg)=r,
  \end{equation*} which implies that
  \begin{equation}
  c_1=\displaystyle\frac{r\cos\theta}{2}=c_2.
  \end{equation}
Consequently, relations  \eqref{eq:15} are given by
  \begin{align}
 & \widetilde{r}(s)=\displaystyle\frac{r\cos\theta}{2}(e^s+e^{-s})=r\cos\theta\cosh(s),
 \\& \widetilde{\theta}(s)=\arcsin\bigg(\displaystyle\frac{e^{2s}-1}{e^{2s}+1}\bigg)=\arcsin(\tanh(s)).
  \end{align}
From \eqref{rcos constant}, we immediately derive
  \begin{equation}\label{eq:20}
\widetilde{r}(s)\cos\widetilde{\theta}(s)=\widetilde{r}(0)\cos\widetilde{\theta}(0)=r\cos\theta\cos(0)=r\cos\theta,
  \end{equation}
Whence, integrating   \eqref{eq:16}  from $0$ to $s_0(\theta)$, we deduce
  \[\Zb(r,\theta)-\Zb(r\cos\theta, 0)=gr\cos\theta\int_0^{s_0(\theta)} \rho_\theta(\widetilde{r}(s), \widetilde{\theta}(s)) ds.\]
Now, from the definitions of $\Zb$ and $\Ub$, we determine the expression for the velocity $w(r,\theta)$, that is,
 \begin{equation}\label{eq:21}
 w(r,\theta)=-\displaystyle\frac{\Omega r\cos\theta}{2}+\frac{1}{2\rho\Omega}\left(\displaystyle\frac{F(r\cos\theta)}{r\cos\theta}+g\int_0^{s_0(\theta)}\rho_\theta(\widetilde{r}(s), \widetilde{\theta}(s))ds\right),
 \end{equation}
 where $F(x):=\Zb(x,0)$ is an arbitrarily smooth function.
 
To derive the expression for pressure $p$,  the first equation of \eqref{eq:11} yields
\begin{align}\label{eq:22}
p_r&= \nonumber -g\rho+\Omega\rho\Ub\cos\theta \\&= -g \rho+\displaystyle\frac{F(r\cos\theta)}{r}+g\cos\theta\int_0^{s_0(\theta)}\rho_\theta(\widetilde{r}(s),\widetilde{\theta}(s))ds   ,
 \end{align}
while from the second one we infer
 \begin{align}\label{eq:23}
 p_\theta&=- \nonumber \Omega r\rho \Ub\sin\theta\\&=-\tan \theta\Bigg[F(r\cos\theta)+gr\cos\theta\int_0^{s_0(\theta)}\rho_\theta(\widetilde{r}(s),\widetilde{\theta}(s))ds\Bigg].
 \end{align}
If we integrate \eqref{eq:22} with respect to r, and use the substitution $y=r\cos \theta$, we obtain
 \begin{equation}\label{eq:24} 
 p(r, \theta)=-g\int_a^r\rho(\xi,\theta)d\xi+\int_{a\cos\theta}^{r\cos\theta}\Bigg[\displaystyle\frac{F(y)} {y}+H(y,\theta)\Bigg]dy+\widetilde{C}(\theta),
 \end{equation}
 where  $a$ is a constant,  $\theta\rightarrow \widetilde{C}(\theta)$ is a smooth function and
 \begin{equation}\label{eq:25}
 H(y,\theta)=\int_0^{s_0(\theta)} g\rho_\theta (y \cosh(s), \widetilde{\theta}(s))ds.
 \end{equation}
 In equation \eqref{eq:24}, differentiating with respect to $\theta$ yields the expression for $p_\theta$,
\begin{align*}
    p_\theta&=-g\int_a^r \rho_\theta(\xi, \theta) d\xi- \tan \theta\left(F(r\cos\theta)-F(a\cos\theta)
\right)\\& \quad +\int_{a\cos\theta}^{r\cos\theta} H_\theta(y,\theta)dy- \sin\theta \left ( r\,H(r\cos\theta,\theta)-aH(a\cos\theta, \theta)\right)+\widetilde{C}'(\theta).
\end{align*}
Since 
 \begin{equation}\label{eq:26}
 H_\theta(y, \theta)=g\rho_\theta(y\cosh (s_0(\theta)), \widetilde{\theta}(s_0(\theta)))\cdot \frac{1}{\cos\theta}, 
 \end{equation}
 and
\begin{equation*}
    \cosh (s_0)=\displaystyle\frac{\widetilde{r}(s_0)}{r\cos\theta}=\displaystyle\frac{1}{\cos\theta},
\end{equation*}
from \eqref{eq:19} we obtain 
 \begin{equation}\label{eq:27}
H_\theta(y,\theta)=g\rho_\theta\bigg(\displaystyle\frac{y}{\cos\theta},\theta\bigg)\cdot s_0'(\theta)=\displaystyle\frac{g}{\cos\theta}\rho_\theta\bigg(\displaystyle\frac{y}{\cos\theta},\theta\bigg).
 \end{equation}Consequently, 
 \begin{equation}\label{eq:28}
\int_{a\cos\theta}^{r\cos\theta}H_\theta(y,\theta)dy=\int_{a\cos\theta}^{r\cos\theta}\displaystyle\frac{g}{\cos\theta}\rho_\theta\bigg(\frac{y}{\cos\theta},\theta\bigg)dy=\int_a^r g\rho_\theta(\widetilde{r}, \theta)d\widetilde{r}.
 \end{equation}
Hence, we derive the expression for $p_\theta$  given by
\begin{align}
   p_\theta \nonumber  &=-F(r\cos\theta)\tan\theta+\tan \theta\,F(a\cos\theta)-r\sin\theta\, H(r\cos\theta, \theta) \\&
   \quad +a\, \sin\theta\, H(a\cos\theta, \theta) +\widetilde{C}'(\theta). \label{eq:29}
\end{align}
From the formulas \eqref{eq:29} and \eqref{eq:23}, we deduce
 \[\widetilde{C}'(\theta)+\tan\theta\,F(a\cos\theta)+
 a\sin\theta\, H(a\cos\theta, \theta)=0,\]
 which gives,
 \begin{equation}\label{eq:30}
 \widetilde{C}(\theta)=A-\int_0^\theta \tan \xi\,F(a\cos\xi) d\xi-a\int_0^\theta \sin\xi \,H(a\cos\xi,\xi) d\xi.
 \end{equation}
From all the previous results, we derive the formula for the pressure,
\begin{align}\label{eq:31}
     p(r, \theta)&=A-g\int_a^r\rho(\xi,\theta)d\xi+\int_{a\cos\theta}^{r\cos\theta}\Bigg[\displaystyle\frac{F(y)}{y}+H(y,\theta)\Bigg]dy\\& \quad -\int_0^\theta \tan \xi \,F(a\cos\xi)d\xi-a\int_0^\theta  \sin\xi  \, H(a\cos\xi,\xi)d\xi \nonumber.
\end{align}
\subsection{On the dynamic conditions}

 The next step is to study the dynamic conditions \eqref{eq:10} on the free surface $r=R+h(\theta)$. The normal vector to a surface  $H(r,\theta, z)=0$ (with $H$ unknown), is given by
 \begin{equation}\label{eq:32}
\vec{ N}=H_r\vec{e}_r+\displaystyle\frac{1}{r}H_\theta \vec{e}_\theta+H_z\vec{e}_z.
\end{equation}  
Given that the kinematic boundary condition on the free surface implies $H(r,\theta,z)=r-R-h(\theta)$, the normal vector is
\[\vec{N}=\vec{e}_r-\displaystyle\frac{h_\theta}{r}\vec{e}_\theta\]
Therefore, the pointing unit normal vector is
\[\vec{n}=\displaystyle\frac{\vec{N}}{\lVert \vec{N} \rVert}=\displaystyle\frac{r}{\sqrt{r^2+{h_\theta}^2}}\cdot \vec{N}=\displaystyle\frac{r}{\sqrt{r^2+{h_\theta}^2}}\vec{e}_r-\displaystyle\frac{h_\theta}{\sqrt{r^2+{h_\theta}^2}}\vec{e}_\theta\]
Moreover, the divergence of $\vec{n}$ is
\begin{align}
    \nabla\cdot \vec{n}&=\displaystyle\frac{1}{r}\partial _r(r n_r)+\displaystyle\frac{1} {r}\partial_\theta(n_\theta)+\partial_z(n_z)\\& 
  =  \displaystyle\frac{1}{r}\bigg(\displaystyle\frac{r} {\sqrt{r^2+h_\theta^2}}+\frac{r h_\theta^2}{(r^2+h_\theta^2)^\frac{3}{2}}\bigg)-\displaystyle\frac{1}{r}\cdot \frac{r^2 h_{\theta\theta}}{(r^2+h_\theta^2)^\frac{3}{2}}\\& \nonumber
  =\frac{r^2+2h_\theta^2}{(r^2+h_\theta^2)^\frac{3}{2}}-\frac{r h_{\theta\theta}}{(r^2+h_\theta^2)^\frac{3}{2}}\\& 
  =
  \frac{(R+h(\theta))^2+2h_\theta^2}{((R+h(\theta))^2+h_\theta^2)^\frac{3}{2}}-\frac{h_{\theta\theta}(R+h(\theta))^2}{((R+h(\theta))^2+h_\theta^2)^\frac{3}{2}}. \label{eq:33}
\end{align}
 From the dynamic boundary condition \eqref{eq:10}, it follows that
\[
 P(\theta)=p(r,\theta)-\sigma\nabla\cdot \vec{n}.
\]
Therefore, we obtain the following Bernoulli-type relation:
\begin{align}
\label{eq:34}P(\theta)&=A-g\int_a^r\rho(\xi,\theta)d\xi+\int_{a\cos\theta}^{r\cos\theta}\Bigg[\displaystyle\frac{F(y)}{y}+H(y,\theta)\Bigg]dy\\& \quad
    -\int_0^\theta  \tan \xi\, F(a\cos\xi) d\xi -\int_0^\theta a\sin\xi \, H(a\cos\xi,\xi)d\xi \nonumber \\& \quad 
    -\sigma\cdot\Bigg[\frac{(R+h(\theta))^2+2h_\theta^2}{((R+h(\theta))^2+h_\theta^2)^\frac{3}{2}}-\frac{h_{\theta\theta}(R+h(\theta))^2}{((R+h(\theta))^2+h_\theta^2)^\frac{3}{2}}\Bigg].
\end{align}
To establish the existence of solutions, we apply the implicit function theorem. In order to accomplish this, firstly it is necessary to non-dimensionalize the equation \eqref{eq:34}. In this way, the physical quantities can be easily compared in a relevant way. Dividing throughout by atmospheric pressure $P_{atm}$ yields the problem
\begin{align}
\label{eq:35}    \Pb(\theta)&= -\frac{g}{P_{atm}}\int_a^{(1+\Hb(\theta))R} 
\rho(\xi,\theta) d\xi+\int_{a\cos\theta}^{(1+\Hb(\theta))R\cos\theta} \Bigg[\frac{F(y)}{y}+H(y, \theta)\Bigg]dy\\& 
\quad -\frac{\sigma}{RP_{atm}} \nonumber
\cdot\Bigg[\frac{(1+\Hb(\theta))^2+2\Hb_\theta^2}{((1+\Hb(\theta))^2+\Hb_\theta^2)^\frac{3}{2}}-\frac{\Hb_{\theta\theta}(1+\Hb(\theta))^2}{((1+\Hb(\theta))^2+\Hb_\theta^2)^\frac{3}{2}}\Bigg]
 +\frac{\widetilde{C}(\theta)}{P_{atm}},
\end{align}
where the non-dimensionalized functions are defined through
\[\Hb(\theta):=\frac{h(\theta)}{R}, \Pb(\theta):=\frac{P(\theta)}{P_{atm}}.
\]

\subsection{Solutions describing the
free surface}
To study the problem \eqref{eq:35}, we represent it  as a functional equation for a more straightforward analytical study, that is, we look for nontrivial solutions of the problem
\begin{equation}\label{ec implicita}
    \Fb(\Hb, \Pb)=0,
\end{equation}
where 
\begin{align}
\label{eq:36}    \Fb(\Hb, \Pb)&:=-\frac{g}{P_{atm}}\int_a^{(1+\Hb(\theta))R} 
\rho(\widetilde{r},\theta) d\widetilde{r}+\int_{a\cos\theta}^{(1+\Hb(\theta))R\cos\theta} \Bigg[\frac{F(y)}{y}+H(y, \theta)\Bigg]dy+\frac{\widetilde{C}(\theta)}{P_{atm}} \\& \quad -\frac{\sigma}{R P_{atm}}
\cdot\Bigg[\frac{(1+\Hb(\theta))^2+2\Hb_\theta^2}{((1+\Hb(\theta))^2+\Hb_\theta^2)^\frac{3}{2}}-\frac{\Hb_{\theta\theta} \nonumber(1+\Hb(\theta))^2}{((1+\Hb(\theta))^2+\Hb_\theta^2)^\frac{3}{2}}\Bigg]-\Pb(\theta).
\end{align}
Clearly, the functional  $\Fb\colon C^2[0, \varepsilon]\times C[0, \varepsilon]\rightarrow C[0,\varepsilon]$
 is a continuous and differentiable mapping. 

 The main goal is to relate variations of the pressure of the free surface to variations of the shape of the free surface. Setting $\Hb\equiv 0$ in \eqref{eq:35}, it defines the situation of an undisturbed free surface, following the curvature of the Earth, away from the Equator. This implies
\begin{equation}\label{eq:37}
\Fb(0,\Pb_0)=0,
\end{equation}
where the surface pressure distribution required to preserve the undisturbed shape is
\begin{equation}\label{eq:38}
\Pb_0(\theta)=-\frac{g}{P_{atm}}\int_a^R \rho(\xi,\theta)d\xi+\frac{1}{P_{atm}}\int_{a\cos\theta}^{R\cos\theta}\Bigg[\frac{F(y)}{y}+H(y,\theta)\Bigg] dy+\frac{\widetilde{C}(\theta)}{P_{atm}}-\frac{\sigma}{R P_{atm}}.
\end{equation}
In the subsequent, we analyze the Fr\'{e}chet derivative of $\Fb$ with respect to the first argument.
One has, 
\begin{flalign*}
    \Fb(s\Hb,\Pb_0)-\Fb(0,\Pb_0)&=-\frac{g}{P_{atm}}\int_R^{(1+s\Hb(\theta))R} \rho(\xi,\theta)d\xi\\&
    \quad+
    \frac{1}{P_{atm}}\int_{R\cos\theta} ^ {(1+s\Hb(\theta))R}\Bigg[\frac{F(y)}{y}+H(y,\theta)\Bigg] dy\\&
  \quad -  \frac{\sigma}{RP_{atm}}\left( J(s\Hb)-J(0)\right),
\end{flalign*}
where \begin{equation*}
    J(\Hb):=\frac{(1+\Hb(\theta))^2+2\Hb_\theta^2}{((1+\Hb(\theta))^2+\Hb_\theta^2)^\frac{3}{2}}-\frac{\Hb_{\theta\theta}(1+\Hb(\theta))^2}{((1+\Hb(\theta))^2+\Hb_\theta^2)^\frac{3}{2}}.
\end{equation*}
From \eqref{eq:24} we find,\begin{equation*}
      \frac{F(y)}{y}+H(y,\theta)=(2 w(r,\theta)+\Omega\cos\theta)\cdot\Omega\rho(r,\theta),
\end{equation*}
whence,
 \[\lim_{s\rightarrow 0} \frac{1}{s}\left(-\frac{g}{P_{atm}}\int_R^{(1+s\Hb(\theta))R}\rho(\xi,\theta)d\xi+\frac{1}{P_{atm}}\int_{R\cos\theta}^{(1+s\Hb(\theta))R}((2 w(r,\theta)+\Omega\cos\theta)\cdot\Omega\rho(r,\theta)) dy\right)\]
 \begin{equation}\label{eq:40}
  =\frac{\rho(R,\theta)}{P_{atm}}[-gR+\Omega R\cos\theta(2w(R,\theta)+\Omega R\cos\theta)]\Hb.
 \end{equation}
Concerning  the most problematic term in obtaining the form of derivative of $\Fb$, we have
\begin{align*}
      \lim_{s\rightarrow 0}\frac{J(s\Hb)-J(0)}{s}&=
    \lim_{s\rightarrow 0} \frac{1}{s}\left(\frac{(1+s\Hb)^2-((1+s\Hb)^2+s^2\Hb_\theta^2)^\frac{3}{2}+2s^2\Hb^2}{((1+s\Hb)^2+s^2\Hb_\theta ^2)^\frac{3}{2}}\right)\\& \quad -\left(\frac{s\Hb_{\theta\theta}(1+s\Hb(\theta))^2}{((1+s\Hb(\theta))^2+s\Hb_\theta^2)^\frac{3}{2}}\right).
\end{align*}
One easily sees that \begin{equation*}
    \lim_{s\rightarrow 0} \frac{1}{s}\left(\frac{s\Hb_{\theta\theta}(1+s\Hb(\theta))^2}{((1+s\Hb(\theta))^2+s\Hb_\theta^2)^\frac{3}{2}}\right)=\Hb_{\theta \theta}.
\end{equation*}
Thus, 
\begin{align*}
  &  \lim_{s\rightarrow 0}\frac{J(s\Hb)-J(0)}{s}=
    \lim_{s\rightarrow 0} \frac{1}{s}\left(\frac{(1+s\Hb)^2-((1+s\Hb)^2+s^2\Hb_\theta^2)^\frac{3}{2}+2s^2\Hb^2}{((1+s\Hb)^2+s^2\Hb_\theta ^2)^\frac{3}{2}}\right)-\Hb_{\theta \theta}
    \\&
    =\lim_{s\rightarrow 0}\frac{(1+s\Hb)^2-((1+s\Hb)^2+s^2\Hb_\theta^2)^\frac{3}{2}}{s}-\Hb_{\theta \theta}\\&
    =\lim_{s\rightarrow 0}\frac{(1+s\Hb)^3-((1+s\Hb)^2+s^2\Hb_\theta ^2)^\frac{3}{2} -s\Hb(1+s\Hb)^2}{s}-\Hb_{\theta \theta}\\&
    =\lim_{s\rightarrow 0} \left(\frac{(1+s\Hb)^3-((1+s\Hb)^2+s^2\Hb_\theta^2)^\frac{3}{2}}{s}\right)-\lim_{s\rightarrow 0} \Hb(1+s\Hb)^2-\Hb_{\theta \theta}\\&
    =-\Hb-\Hb_{\theta \theta}+\lim_{s\rightarrow 0} \frac{(1+s\Hb)-\sqrt{(1+s\Hb)^2+s^2\Hb_\theta^2}}{s}\\& \quad \cdot\lim_{s\rightarrow 0}\left((1+s\Hb)^2+(1+s\Hb)\sqrt{(1+s\Hb)^2+s^2\Hb_\theta ^2}+(1+s\Hb)^2+s^2\Hb_\theta^2\right).
\end{align*}
Since
\[\lim_{s\rightarrow 0} \frac{(1+s\Hb)-\sqrt{(1+s\Hb)^2+s^2\Hb_\theta^2}}{s}=0,\]
and 
\[\lim_{s\rightarrow 0}\left((1+s\Hb)^2+(1+s\Hb)\sqrt{(1+s\Hb)^2+s^2\Hb_\theta ^2}+(1+s\Hb)^2+s^2\Hb_\theta^2\right)
=3,\]
we infer 
\begin{equation}\label{eq:43}
\lim_{s\rightarrow 0}\frac{J(s\Hb)-J(0)}{s}=-\Hb-\Hb_{\theta\theta}.
\end{equation}
Consequently,  relations \eqref{eq:40} and \eqref{eq:43} yields
\begin{equation}\label{eq:44}
\Db_\Hb \Fb(0, \Pb_0)(\Hb)=\frac{\rho(R,\theta)}{P_{atm}}\big[-gR+\Omega R\cos\theta(2w(R, \theta)+\Omega R\cos\theta)\big]\cdot \Hb+\frac{\sigma}{P_{atm}}\big(\Hb_{\theta\theta}+\Hb\big).
\end{equation}
Let us denote $$\gamma (\theta):=\displaystyle\frac{\rho(R,\theta)}{P_{atm}}\left(-gR+\Omega R\cos\theta(2w(R, \theta)+\Omega R\cos\theta)\right)+\frac{\sigma}{P_{atm}} $$ and $d:=\displaystyle\frac{\sigma}{P_{atm}}$.
Hence, the derivative $\Db_\Hb \Fb(0, \Pb_0)(\Hb)$  has the representation
\begin{equation}\label{ec diff principala}
    \Db_\Hb \Fb(0, \Pb_0)(\Hb)=d\Hb_{\theta\theta}+\gamma\,\Hb.
\end{equation}
~\\
Now, we are ready to state the main result of this paper.
\begin{theorem}\label{teorema homomorfism}
 Let $X$ be the space of continuous  functions on $[0,\varepsilon]$  having values and derivatives equal to zero at the point $0$ , i.e., \[X =\{u \in C^2[0, \varepsilon]: u(0)=0 \text{ and }u^\prime(0)=0\}.\]
Then,  the operator $\Db_\Hb \Fb(0, \Pb_0):X\rightarrow C[0,\varepsilon]$ is a linear homeomorphism.

\end{theorem}


\begin{proof}
The continuity of the operator $\Db_\Hb \Fb(0, \Pb_0)$ is straightforward. To complete the proof, we need to establish the bijectivity. This suffices for our purpose since, according to the Bounded Inverse Theorem, any linear, continuous, and bijective operator between two Banach spaces is a homeomorphism (see, e.g., \cite{inverse_theorem, kothe1983}).

To proceed with the proof of the bijectivity, we may use either \Cref{teorema_existenta_solutie_si_reprezentare} or \Cref{th coddington}. 

\textbf{1) Use of \Cref{teorema_existenta_solutie_si_reprezentare}.}
To establish the bijectivity of $\Db_\Hb \Fb(0, \Pb_0)$, it suffices to demonstrate that for any $h\in C([0,\varepsilon])$, there exists a unique $u_h\in X$ such that $\Db_\Hb \Fb(0, \Pb_0)(u_h)=h$. Note that, this is equivalent with proving that the second order differential equation   \begin{equation}\label{eq:46}
     \begin{cases}
         u^{\prime\prime}(\theta)+\frac{\gamma(\theta) }{d}u(\theta)=\frac{\varphi(\theta)}{d}, \quad \text{ on }[0,\varepsilon] \\
        u(0)=0 \\
        u^\prime(0)=0,
  \end{cases}
 \end{equation}
 has a unique solution. Following the conventional approach often employed in the literature for second-order ordinary differential equations, we can represent \eqref{eq:46} as a system of first order differential equations
  \begin{equation}\label{sistem forma vector}
   \begin{cases}
         Y'+AY=B_\varphi, \text{ on }[0,\varepsilon],\\
         Y(0)=O_2,
   \end{cases}
 \end{equation}
 where $O_2$ is the zero vector from $\mathbb{R}^2$,  \[Y=\begin{bmatrix}
     u'\\
     u
 \end{bmatrix}, 
 A=\begin{bmatrix}
     -

1 & 0\\
     0 & \frac{\gamma}{d}
\end{bmatrix} \text{ and }
B_\varphi=\begin{bmatrix}
    0\\
    \frac{\varphi}{d}
\end{bmatrix}.\]
Since $A$ and $B_h$ are continuous functions, from \Cref{teorema_existenta_solutie_si_reprezentare}, there exists a unique solution  \begin{equation*}
    Y(t)=\int_0^t e^{\int_s^tA(\xi)d\xi} B_\varphi(s) ds,
\end{equation*} to the problem \eqref{sistem forma vector}, as desired.

\textbf{2) Use of \Cref{th coddington}.} We prove that $\Db_\Hb \Fb(0, \Pb_0)$ is both injective and surjective.

\textit{Check of the injectivity:} Let $u$ be an element from X such that 
\[\Db_\Hb \Fb(0, \Pb_0)(u)=0,\]
and let $\varPhi_1, \varPhi_2$ be a basis of solutions for the  second order differential equation \eqref{ec diff principala}. Thus,  we may write 
$u=c_1\varPhi_1+c_2\varPhi_2$, for some $c_1$ and $c_2$ real numbers. 
Since we are looking  for solutions with  $u(0)=u^{\prime}(0)=0$,  the the linear independence of $\varPhi_1$ and $\varPhi_2$ implies $c_1=c_2=0$, hence $u\equiv 0$.

\textit{Check of the surjectivity:}
Let $g\in C[0, \varepsilon]$. Following \Cref{th coddington}, there exists $\widetilde{u}$ from $C^2[0, \varepsilon]$ such that 
\begin{equation}\label{eq:452}
\Db_{\Hb}\Fb(0, \Pb_0)(\widetilde{u})=g,
\end{equation}
with the form $\widetilde{u}=u_p+c_1\varPhi_1+c_2\varPhi_2$, 
where
\[u_p(\theta)= \int_0^{\theta}\frac{ g(s) \left(\varPhi_1(\theta) W_1(s) +\varPhi_2(\theta) W_2(s)\right)}{W(\varPhi_1, \varPhi_2)(s)} ds.\]
To find the constants $c_1, c_2$, we impose $u(0)=u^\prime(0)=0$. Thus, 
 
\[0=c_1\varPhi_1(0)+c_2\varPhi_2(0)\]
and, by taking heed of the fact that $W_1(\theta)=-\varPhi_2(\theta)$ and $W_2(\theta)=\varPhi_1(\theta)$, for all $\theta\in[0,\varepsilon]$,
\[0=u'(0)=c_1\varPhi_1^\prime(0)+c_2\varPhi_2^\prime(0).\]
Due to the linear independence of $\varPhi_1$ and $\varPhi_2$, we obtain, again, that $c_1=c_2=0$. Thus, the unique solution of $\eqref{eq:452}$ will be 
\begin{equation}\label{eq:453}
    u(\theta)=u_p(\theta).
\end{equation}
Since $\varPhi_1W_1+\varPhi_2W_2=0$, we find that
\[u^{\prime}(\theta)=\sum_{k=1}^2\varPhi_k^\prime(\theta) \int_0^{\theta}\frac{W_k(s) g(s)}{W(\varPhi_1, \varPhi_2)(s)} ds ,\]
and
\begin{equation}\label{eq:454}
    u^{\prime\prime}(\theta)=\sum_{k=1}^2\left(\varPhi_k^{\prime\prime}(\theta)\int_0^{\theta}\frac{W_k(s) g(s)}{W(\varPhi_1, \varPhi_2)(s)} ds+\frac{\varPhi_k^{\prime}(\theta)W_k(\theta)g(\theta)}{W(\varPhi_1, \varPhi_2)(\theta)}\right),
\end{equation}
for all $\theta\in [0, \varepsilon].$ With this conclusion, we finish our proof.

\end{proof}
As a consequence of \Cref{teorema homomorfism}, we obtain the following existence result. \begin{theorem}\label{teorema principala}
    For small enough variations of $\Pb$, there exists $\Hb \in X$ such that \eqref{ec implicita} holds true.
\end{theorem}
\begin{proof}
    The conclusion follows immediately from \Cref{teorema functiei implicie} and \Cref{teorema homomorfism}.
\end{proof}


\section*{Appendix}
This section presents well-known results from the literature used throughout this paper. The primary result (\Cref{teorema principala}) relies on the Implicit Function Theorem (see, e.g., \cite{functionala}).
\begin{theorem}\label{teorema functiei implicie}
    Let $X,Y,Z$ be Banach spaces,  $U\subset X\times Y$ an open neighbourhood of a point $(x_0,y_0)\in X\times Y$ and let $f\colon U \to Z$ be a continuous functions. Assume that: \begin{itemize}
        \item[i) ] The function $f$ satisfies $f(x_0,y_0)=0$.
        \item[ii)] The partial derivative  $f_y(x_0,y_0)$ exists and is an linear  homeomorphism from $Y$ to $Z$.
    \end{itemize}
    Then, there exists an open neighbourhood $U_1$ of $x_0$ and a unique  $g\colon U_1 \to Y$ continuous function such that $g(x_0)=y_0$ and $f(x,g(x))=0$ on $U_1$.
\end{theorem}

Next, we  focus on some  existence and uniqueness results for a differential equation. Let us consider the Cauchy problem \begin{equation}\label{sistem_cauchy_general}
    \begin{cases}
        X^\prime= A(t) X(t)+B(t)\\X(0)=O_n
    \end{cases} \text{~on~} [0,T],
\end{equation}
where $$X\colon \mathbb{R}\to \mathbb{R}^n, A,B\in \mathcal{M}_{n\times n}(C[0,T])$$ are unknowns, and $O_n$ is the zero vector from $\mathbb{R}^n$.
These equations have been extensively studied, and their behavior is well-established. The following result is a classic in the theory of differential equations and refers to the exponential-like representation of solutions for the system \eqref{sistem_cauchy_general}. We send to   \cite{magnus, 1968, p} for further details.  \begin{theorem}\label{teorema_existenta_solutie_si_reprezentare}
    If the matrix $A(t)$ commute with $\int_0^t A(s)ds$, i.e., 
\begin{equation*}
    A(t) \int_0^t A(s)ds =\int_0^t A(s)ds A(t),\text{ for all }t\in [0,T],
\end{equation*}
then the system \eqref{sistem_cauchy_general} has  one  $C^1$ solution and \begin{equation*}
    \Phi(t):=\int_0^t A(s)ds, 
\end{equation*}
is a fundamental matrix, $\text{ for all }t\in [0,T]$. Moreover, the solution is given by\begin{equation*}
    X(t)=\int_0^t e^{\int_s^t A(\xi)d\xi} B(s)ds, \text{ for all } t\in [0,T].
\end{equation*} 
\end{theorem}
A similar existence result is concerned with the representation of the solution for a general linear differential equation of arbitrarily order $p>0$ on some interval $I$. We consider the linear differential operator \begin{equation*}
    L(u)(x)=\sum_{i=0}^p a_i(x)u^{(i)}(x), \quad x\in I,
\end{equation*}
where $a_{i}$ are continuous functions on $I$, and $u$ of $C^p(I)$ class. 
\begin{theorem}[{\cite[Chapter~3, Th. 11]{Coddington}}]\label{th coddington}
    Let $b$ be a continuous function on $I$. Then, every solution of the equation $Lu=b$ can be written as \begin{equation*}
        u=u_p+\sum_{i=0}^p c_i \phi_i,
    \end{equation*}
    where  $\{\phi_i\}_{0\leq i\leq p}$  is a basis for the solutions of $Lu=0$ and $u_p$ is a particular solution of $Lu=b$. Additionally, we may look for  $u_p$ in the form \begin{equation*}
        u_p=\sum_{i=0}^p \phi_i \int_{x_0}^x \frac{W_k(t)b(t)}{W(t)},
    \end{equation*}
    where $W$ is the Wronskian of the  basis $\{\phi_i\}_{0\leq i\leq p}$, and $W_k$ is the Wronskian obtained by replacing the $k$th column of $W$ with $(0,\ldots,0,1)$.
\end{theorem}

\end{document}